\documentclass[preprint,12pt,3p]{elsarticle}
\usepackage{verbatim}
\usepackage{amsmath,amsthm,amsfonts,amssymb}
\usepackage{algorithm}
\numberwithin{equation}{section}
\usepackage{natbib}
\usepackage{mathrsfs}
\def\BState{\State\hskip-\ALG@thistlm}
\makeatother
\usepackage{mathtools}
\usepackage{comment}
\usepackage{cases}

\usepackage{hyperref}
\usepackage[normalem]{ulem}
\usepackage{xcolor,sectsty}
\usepackage{subfigure}
\usepackage[active,new,old]{correct} 
\newtheorem{theorem}{Theorem}[section]
\newtheorem{lemma}[theorem]{Lemma}
\newtheorem{definition}[theorem]{Definition}
\newtheorem{example}[theorem]{Example}

\newtheorem{corollary}[theorem]{Corollary}
\newtheorem{proposition}[theorem]{Proposition}
\journal{Arxiv}

\newcommand{\dg}{{\dagger}}

\newcommand{\rg}{{\mathscr{R}}}
\newcommand{\nl}{{\mathscr{N}}}
\newcommand{\ep}{\scriptsize\mbox{\textcircled{$\dagger$}}}
\newcommand{\core}{\scriptsize\mbox{\textcircled{\#}}}
 \def\cnm{{\mathbb C}^{n\times m}}
\def\cnn{{\mathbb C}^{n\times n}}
\def\cmm{{\mathbb C}^{m\times m}}
\def\cmn{{\mathbb C}^{m\times n}}

\begin{document}

\begin{frontmatter}
\title{ 
{\bf  Characterization and Representation of Weighted Core Inverse of  Matrices
}}

\author{ Ratikanta Behera$^\dag$$^a$,  Jajati Keshari Sahoo$^*$,  R. N. Mohapatra $^\dag$$^b$}
\vspace{.3cm}

\address{  
               $^{\dag}$Department of Mathematics,\\
                  University of Central Florida, Orlando, USA.\\
                        \textit{E-mail$^a$}: \texttt{ratikanta.behera@ucf.edu}\\
                        \textit{E-mail$^b$}: \texttt{ram.mohapatra@ucf.edu}
                        
                        \vspace{.3cm}
                        $^{*}$Department of Mathematics,\\
                       Birla Institute of Technology $\&$ Science, Pilani K.K. Birla Goa Campus, India
                        \\\textit{E-mail}: \texttt{jksahoo\symbol{'100}goa.bits-pilani.ac.in}
         
                        }

\begin{abstract}
In this paper, we introduce new representation and characterization of the weighted core inverse of matrices. Several properties of these inverses and their interconnections with other generalized inverses are also explored. Through one-sided core and dual-core inverses, we have shown the existence of a generalized weighted Moore-Penrose inverse of matrices. Further, by applying a new representation and using the properties of the weighted core inverse of a matrix, we discuss some new results related to the reverse order law for these inverses.
\end{abstract}


\end{frontmatter}

\section{Introduction}
\subsection{Background and motivation}
Baksalary and Trenkler \cite{BakTr10, BakTr14} introduced the core inverse of a square matrix and discussed the existence of such matrices. However, the right weak generalized inverse (see \cite{BenIsrael03, Cline68}) was proposed earlier, which {was} renamed as the core inverse. In the literature, many authors have discussed the core inverse (see \cite{ GaoChen18, HaiTing19,WangLi19,zhou2019core}), and the relations with other generalized inverses \cite{kurata,rakic, wang2015}. Prasad and Mohana \cite{ManMo14} have discussed {core EP} inverse of a square matrix of an arbitrary index. Then the authors of \cite{fer18} extended the notion of core EP inverse to weighted core EP inverse and studied some properties of these weighted core EP inverses. Subsequently, several characterizations of the $W$-weighted core-EP inverse are discussed in  \cite{gao2018}. Then, Mosi{\'c} \cite{mosic2019} studied the weighted core-EP inverse of an operator between Hilbert spaces. Further, many results are available in elements of rings with involution \cite{DijanaChu18,rakic2015, ZhangXu17}. {In addition to these, the authors of \cite{MaH19} provided a characterization of the weighted core-EP inverse in terms of matrix decomposition.} Tian and Wang have discussed weighted-EP matrices as a generalization of EP matrices (see \cite{TianWang11}). {Among the weighted generalized inverses, the generalized
weighted Moore-Penrose inverse is also important, as it includes the weighted Moore-Penrose inverse, Moore-Penrose inverse, as well as an ordinary matrix inverse. Several characterizations and representations of the generalized inverses for matrices are discussed in \cite{RaoMitra71, khatri1968generalized}. Further, these inverses found their use in many applications, including statistics \cite{RaoRao98,Katri68}. Analogous to the generalized inverses of a commutative ring, the authors of \cite{BapatRao90,PrasadBapa91} have discussed a necessary and sufficient condition for the existence of a generalized inverse in terms of the minors for a matrix over an integral domain. In this context, the notion of the weighted Moore-Penrose inverse was introduced by Prasad and Bapat  in \cite{PrasadBapat92} and showed that the invertibility of $M$ and $N$  matrices were sufficient for the existence of that inverse.} It is worth mentioning that the weighted inverses of matrices have been recently investigated in \cite{PredragKM17, PreMo20, ZhangWei16}. The vast literature on the weighted inverses and its multifarious extensions and applications in different areas of mathematics, motivate us to introduce a new representation and characterization of the weighted core inverse of matrices.

A summary of the main points of the discussion is mentioned below:
\begin{enumerate}
\item[(1)] The {definitions} of $M$-weighted core and $N$-weighted dual-core inverses are introduced. Using these definitions, some characterizations and {representations} of weighted core inverse of matrices are investigated.

\item[(2)] Conditions for the existence and representations of generalized weighted Moore-Penrose inverse are discussed using the weighted core inverse of a matrix.

\item[(3)] {Various properties of weighted core inverse of matrices and their relationships with other generalized inverses are demonstrated.}

\item[(4)] A few necessary and sufficient conditions for reverse order law to hold for weighted core inverse of matrices are also established. 
\end{enumerate}

The purpose of this paper is also to obtain conditions  {when the group inverse, weighted core inverse, generalized weighted Moore-Penrose inverse are the same.} The beauty of these representations 
%these inverses 
is that the algorithm developed for one particular type of {generalized inverse} can be easily extended to the other type of generalized inverses. This gives us the flexibility to choose an appropriate generalized inverse depending on the  applications at hand.

On the other hand, the reverse order law for a generalized inverse plays an important role in the theoretical research and numerical computations in many areas (see \cite{RaoMitra71, WangYimin18, SunWei02}). For example, consider two {invertible} matrices $A$ and $B$. The equality $(AB)^{-1}=B^{-1}A^{-1}$ is called the reverse order law, which is always true for {the} invertible matrices. When the ordinary inverse is replaced by generalized inverse this reverse order law is not true in general \cite{BenIsrael03}. In this context, Greville \cite{Grev1966} first discussed  a necessary and sufficient condition for this equality in the framework of the Moore-Penrose inverse in 1966. Then Baskett and Katz \cite{baskett1969} studied the reverse order law for $EP_r$ matrices.  Deng \cite{deng2011} investigated {it for }the group inverse. Since then, a significant number of papers on the reverse order law for various classes of generalized inverses have been studied (see \cite{MosiDij12, MosiDj11, PaniBeheMi20, JR_rev, YiminWei98}). In this context, we focus our attention on discussing a few necessary and sufficient conditions for the reverse order law to hold for weighted core inverse of matrices. 
\subsection{Outline}
The remainder of this paper is organized as follows. 
Some necessary basic definitions and notations are presented in Section 2. In Section 3, we introduce a new expression for weighted core and the dual-core inverse of a matrix. In addition, we discuss representation and characterization of weighted core inverse and generalized weighted Moore-Penrose inverse for matrices.  Finally some results on {reverse order law are discussed in Section 4. In Section 5, we give our concluding remarks.}

\section{Preliminaries}
\subsection{Notations and definitions}
Let $ \mathbb{C}^n $ be the $n$-dimensional complex vector space and $\mathbb{C}^{m\times n }$ be the set of $m \times n$ matrices over complex field $\mathbb{C}$. For a matrix $A \in \mathbb{C}^{m \times n}$, let  $\rg(A), \nl(A)$, and $A^*$ denote the range space, null space and conjugate transpose of $A$, respectively. The index of $A\in \mathbb{C}^{n\times n }$ is defined as the smallest non-negative integer $k$, for which $rank(A^k)=rank(A^{k+1})$ and is denoted by $ind(A)$. In particular, if $k=1$, then the matrix $A$ is called index one matrix or core matrix or group matrix.

We now recall the Moore-Penrose inverse \cite{BenIsrael03,RaoMitra71} of a matrix $A \in \mathbb{C}^{m \times n}$. The unique matrix $X\in \mathbb{C}^{n\times m}$, satisfying 
\begin{equation*}
\left(1\right)~AXA=A,~~(2)~XAX=X,~~(3)~(AX)^*=AX,~~(4)~(XA)^*=XA
\end{equation*}
is called the Moore-Penrose inverse  of $A$ and is denoted by $A^\dg$. Further, the definition of weighted Moore-Penrose inverse is recalled.
\begin{definition}\label{wmpi}
Let $A\in\cmn$, $M\in \cmm$ and $N\in \cnn$ be two {invertible hermitian matrices}. A matrix $X\in\cnm$ which  satisfies  
$$
(1)~AXA=A,~~(2)~XAX =X{,}~~
\left(3^M\right)~(MAX)^* = MAX,~~\left(4^N\right)~(NXA)^* = NXA,
$$
is called a { generalized weighted Moore-Penrose inverse} of $A$ and is denoted by $A^{\dagger}_{M,N}$.
\end{definition}
Consider $A \in \cnn$ with $ind(A) = k$. A matrix $X \in \cnn$ satisfying 
\begin{equation*}
\left(1^k\right)~XA^{k+1}=A^k,~~ (2)~XAX =X,~~(5)~AX= XA,
\end{equation*}
is called  the Drazin inverse of $A$ and 
is denoted by $A^D$. The Drazin inverse was first introduced in the context of associative rings and semigroups in 1958 (See \cite{Drazin58}). In particular, for $A \in \cnn$ with $ind(A) = 1$, if a matrix $X \in \cnn$ satisfies 
\begin{equation*}
(1)~AXA=A,~~ (2)~XAX =X,~~(5)~AX= XA,
\end{equation*}
then $X$ is called  the group inverse of $A$ and 
is denoted by $A^{\#}$.  For {the} convenience, we use the notation $A^{(\lambda)}$ for an element of $\{\lambda \}$-inverses of $A$ and $A\{\lambda\}$ for the class of $\{\lambda \}$-inverses of $A$, where $\lambda\in \{1, 1^k, 2, 3, \cdots 3^M,4,4^N,5 \cdots \}$. For instance, a matrix $X \in \cnm$ is called a {\it $\{1\}$-inverse} of
$A \in \cmn$ if $X$ satisfies (1), i.e., $AXA=A$ and we denote $X$ by $A^{(1)}$.

{We next present the definition of core-EP inverse of matrices. Indeed, the definition of core-EP inverse was first introduce by Prasad and Mohana \cite{ManMo14}, and then the authors of \cite{GaoChen18} gave a new characterization for these inverses in terms of three equations} 
\begin{definition}\cite{GaoChen18}\label{c-ep}
 Let  $A \in \cnn$ with $ind(A)=k$. A matrix $X\in \cnn$ satisfying 
 \begin{equation*}
\left(6^k\right)~XA^{k+1}=A^k,~~(7)~ AX^2=X,~~(3)~(AX)^*=AX,
 \end{equation*}
 is called the core-EP inverse of $A$ and denoted by $A^{\ep}$
  \end{definition}
  In particular $k=1$ and for $A \in \cnn$ with $ind(A) = 1$, if a matrix $X \in \cnn$ satisfies 
\begin{equation*}
(6)~XA^2=A,~~ (7)~AX^2=X,~~(3)~(AX)^*= AX,
\end{equation*}
then $X$ is called  the core inverse of $A$ and denoted by $A^{\core}$. One can observe that, the core-EP inverse is unique and  $A^{\core}\in A\{1,2\}$. The extension to the core EP inverse of a rectangular matrix was recently introduced in \cite{fer18}, which is defined next.
\begin{definition}\cite{fer18}\label{wcore-ep}
 Let $A \in \mathbb{C}^{m \times n}$ , $W \in \mathbb{C}^{n \times m}$ and $k=\max\{\ ind (AW),ind (WA)\}$.  A matrix $X\in \mathbb{C}^{m \times n}$ which satisfies 
 \begin{equation*}
  WAWX=(WA)^k[(WA)^k]^{\dagger},~~   \rg(X) \subseteq \rg((AW)^k),
 \end{equation*}
 is called the $W$-weighted core-EP inverse of $A$ and is denoted by $A^{\core,W}$. 
 \end{definition}
 Subsequently, the authors of \cite{gao2018} in their Theorem 2.2 discussed the new expression for the $W$-weighted core inverse, as follows.
\begin{definition}
 Let $A \in \mathbb{C}^{m \times n}$ , $W \in \mathbb{C}^{n \times m}$ and $k=\max\{\ ind (AW),ind (WA)\}$.  A matrix $X\in \mathbb{C}^{m \times n}$ satisfying 
 \begin{equation}\label{eqn1.1c}
  XW(AW)^{k+1}=(AW)^k,~A(WX)^2=X,~(WAWX)^*=WAWX,
 \end{equation}
 is called the $W$-weighted core-EP inverse of $A$. 
\end{definition}
The authors of \cite{gao2018} have also discussed, the unique matrix $X = A\left[(WA)^{\ep}\right]^2$ which satisfies the equations \eqref{eqn1.1c}.  The idea of the following definition is borrowed from \cite{TianWang11} where the authors proved it in their Theorem 3.5. 
\begin{definition}
Let $M$, $N\in \cnn$ be two {invertible hermitian matrices} and $A\in \cnn$ be a core matrix. Then $A$ is called weighted-EP with respect to $(M,N)$ if  $A^\dagger_{M,N}$ exists and $A^\dagger_{M,N}=A^{\#}$.  
\end{definition}

In the spirit of the core inverse of matrices, Sahoo, et al. recently introduced the core inverse and core EP inverse of {tensors} \cite{JR_rev, SahBe20} via the Einstein product. Indeed, tensors are multidimensional generalizations of vectors and matrices. {Now, we recall the following result from \cite{RaoMitra71}.}
{\begin{lemma}\label{lem1.2}
Let  $A\in\cnn$. If $A=A^2X=YA^2$ for some $X,~Y\in\cnn$. Then $A^{\#}$ exist and $A^{\#}=AX^2=YAX=Y^2A$. 
\end{lemma}
}
{
Next, we state the rank cancellation rule which was introduced by Marsaglia and  Styan \cite{mat}. See also Rao and Bhimasankaram (Theorem 3.5.7, \cite{raob}).
\begin{lemma}\label{rtcan}
Let $A,B,C$ and $D$ be the matrices where the product is well defined. If $CAB = DAB$ and $rank(AB) = rank(A)$, then $CA = DA$.
\end{lemma}
}
\section{One-sided weighted core and dual core inverse}
In this section, we first introduce $M$-weighted core and  $N$-weighted dual core inverse for matrices. Then we discuss a few new {representations  and  characterizations} of these inverses.
\begin{definition}\label{mcore}
Let $M\in \cnn$ be an {invertible hermitian matrix} and $A \in \mathbb{C}^{n \times n}$. A matrix $X\in\cnn$ satisfying 
$$\left(3^M\right)~(MAX)^* = MAX,~~(6)~XA^2 = A,~~(7)~AX^2 = X,
$$
is called the $M$-weighted core inverse of $A$ and is denoted by $A^{\core,M}$.
\end{definition}
\begin{example}\rm
Let $A= \begin{bmatrix}
1 & 1\\
0 & 0
\end{bmatrix}$ and {$M=\begin{bmatrix}
3 & i\\
-i & 1
\end{bmatrix}$.} Then we can easily verify that \\the matrix 
{$X= \begin{bmatrix}
1 & i/3\\
0 & 0
\end{bmatrix}$} satisfies the conditions of the Definition \ref{mcore}.
\end{example}
Following the above definition, one can prove the result mentioned below.

\begin{proposition}\label{pro2.2}
For $A \in \mathbb{C}^{n \times n}$, if $X\in A\{6,7\}$, then $X\in A\{1,2\}$.
\end{proposition}
Using Proposition \ref{pro2.2}, we can show the uniqueness of $M$-weighted core inverse which is mentioned in the next theorem.
\begin{theorem}\label{unimcore}
Let $M\in \cnn$ be a  hermitian invertible matrix and $A \in \mathbb{C}^{n \times n}$. Then the $M$-weighted core inverse of $B$ is unique.
\end{theorem}
\begin{proof}
Suppose there exist two $M$-weighted core inverses {of $B$} say $U$ and $V$. Then applying Proposition  \ref{pro2.2}, we obtain 
\begin{eqnarray*}
V&=&VBV=VBUBV=VBUM^{-1}MBV=VBUM^{-1}(MBV)^*\\
&=&VBUM^{-1}V^*B^*M=VM^{-1}MBUM^{-1}V^*B^*M\\
&=&VM^{-1}U^*B^*V^*B^*M=VM^{-1}U^*B^*M=VBU\\
&=&VB^2U^2=BU^2=U.
\end{eqnarray*}
\end{proof}

\begin{definition}\label{ncore}
Let $N\in \cnn$ be an {invertible hermitian matrix} and $A \in \mathbb{C}^{n \times n}$. Then a matrix $X\in\cnn$ satisfying 
$$\left(4^N\right)~(NXA)^* = NXA,~~(8)~A^2X = A,~~(9)~X^2A = X,
$$
is called $N$-weighted dual core inverse of $A$ and is denoted by $A^{N,\core}$.
\end{definition}
\begin{example}
Let  $A= \begin{bmatrix}
1 & 1\\
0 & 0
\end{bmatrix}$, {$N=\begin{bmatrix}
2 & i\\
-i & 2
\end{bmatrix}$.} We can easily verify that \\

\vspace{.3cm}
the matrix 
{$Y= \begin{bmatrix}
0.5-0.25i & 0.5-0.25i\\
0.5+0.25i & 0.5+0.25i
\end{bmatrix}$} satisfy the conditions of Definition \ref{ncore}.\\
\end{example}

Analogous to Proposition \ref{pro2.2}, the following result also {holds}.
\begin{proposition}\label{pro2.5}
For any matrix $A \in \mathbb{C}^{n \times n}$. If the matrix $X\in A\{8,9\}$, then $X\in A\{1,2\}$.
\end{proposition}

The next theorem shows the uniqueness of the $N$-weighted dual core inverse.
\begin{theorem}\label{unincore}
Let $N\in \cnn$ be a  hermitian invertible matrix and $A \in \cnn$. Then the $N$-weighted dual core inverse of $A$ is unique.
\end{theorem}
\begin{proof}
Suppose there exist two $N$-weighted dual core inverses {of $A$} say $X$ and $Y$. Then applying Proposition  \ref{pro2.5}, we obtain 
\begin{eqnarray*}
Y&=&YAY=YAXAY=N^{-1}(NYA)N^{-1}(NXA)Y=N^{-1}A^*Y^*A^*X^*NY\\
&=&N^{-1}A^*X^*NY=N^{-1}(NXA)^*Y=N^{-1}NXAY=XAY\\
&=&(X^2A)AY=X^2(A^2Y)=X^2A=X.
\end{eqnarray*}
This establishes the uniqueness of the dual core inverse of a matrix.
\end{proof}
In view of the Definition \ref{mcore}, we have the following theorem which gives an equivalent characterization with other generalized inverses. 
\begin{theorem}\label{thm2.7}
Let $M \in \cnn$ be an {invertible} hermitian matrix and $A\in \cnn$. Then the following five conditions are equivalent:
\begin{enumerate}
\item [(a)] there exists $X\in \cnn$ such that $AXA=A$, $\rg(X)=\rg(A)$, $\rg(X^T)=\rg((A^*M)^T)$;
\item [(b)] $AXA=A,$  $XAX=X$,  $XA^2=A$, $AX^2=X$, and $(MAX)^*=MAX$;
\item[(c)] $(MAX)^*=MAX$, $XA^2=A$, and $AX^2=X$;
\item[(d)] $A^{\#}$ exists and $X\in A\{1,3^M\}$;
\item[(e)] there exist unique idempotent matrices $P,~Q \in \cnn$ such that $MP=(MP)^*$, $\rg(P)=\rg(A)=\rg(Q)$ and $\rg(Q^T)=\rg(A^T)$.
\end{enumerate} 
If any one of the above relations viz. $(a)-(e)$ are true, then the $M$-weighted core inverse of $A$ is given by $A^{\core,M} =X= QA^{(1)}P=A^{\#}AA^{(1,3^M)}$. Further $P=AA^{\core,M}$, $A^{\#}=\left(A^{\core,M}\right)^2A$, and $Q=A^{\#}A=A^{\core,M}A$.
  \end{theorem}
  
  \begin{proof}
$(a) \Rightarrow (b)$ Let $\rg(X)=\rg(A)$. Then there exist $V \in \cnn$ such that $X=AV$. Using $AXA=A$, we obtain 
\begin{center}
 $X=AV=AXAV=AX^2$.   
\end{center}
 Similarly from $\rg(X^T)=\rg(A^*M)^T$, we have $X=ZA^*M=Z(AXA)^*M=ZA^*X^*A^*M=ZA^*MM^{-1}X^*A^*M=XM^{-1`}X^*A^*M$ for some $Z\in\cnn$. Which further implies $MAX=MAXM^{-1}X^*A^*M$. Thus 
 \begin{center}
     $(MAX)^*=(MAXM^{-1}X^*A^*M)^*=M^*AXM^{-1}X^*A^*M=MAX$
 \end{center}
  since $M$ is hermitian and invertible. Now  
  \begin{center}
    $X=XM^{-1}X^*A^*M=XM^{-1}(MAX)^*=XM^{-1}MAX=XAX$.  
  \end{center}
Again from $\rg(X)=\rg(A)$, we have $A=XV$ for {some $V\in \cnn$.} This yields 
\begin{center}
    $A=XV=XA(XV)=XA(A)=XA^2$.
\end{center}
$(b)\Rightarrow (c)$ It is trivially true.\\
$(c)\Rightarrow (d)$ It is enough to show $A^{\#}$ exists. Let $Y=X^2A$.  From Proposition \ref{pro2.2}, it follows that $AXA=A$. Now $AYA=AX^2A^2=XA^2=A$, $YAY=X^2A(AX^2)A=X^2AXA=X^2A=Y$, and $AY=AX^2A=XA=XXA^2=YA$. Hence $A^{\#}=Y=X^2A$.\\
$(d)\Rightarrow (e)$ Let $Z= A^{\#}AX$, where $X \in A\{1,3^M\}.$ If we choose $P=AZ=AX$ and $Q=ZA= A^{\#}A$, then it is easy to verify that $P$ and $Q$ both are {idempotents}. One can also observe that $MP=MAX$ is a hermitian matrix, $\rg(P)=\rg(A)=\rg(Q)$. From $Q=ZA$ and $A=AXA=AA^{\#}AXA=AZA=AQ$, we obtain $\rg(Q^T)=\rg(A^T)$. Next we will claim the uniqueness of the idempotent matrices $P$ and $Q$. Suppose there exist another pair $P_1$ and  $Q_1$ which satisfies $(e)$. Then $\rg(P)=\rg(A)=\rg(P_1)=\rg(Q)=\rg(Q_1)$ and $\rg(Q^T)=\rg(Q_1^T)$. From $\rg(Q)=\rg(Q_1)$ and $\rg(Q^T)=\rg(Q_1^T)$ we obtain 
\begin{center}
   $Q=Q_1U=Q_1Q_1U=Q_1Q=VQQ=VQ=Q_1, ~{\textnormal{for some}}~U,~V\in\cnn$. 
\end{center}
Similarly, from $\rg(P)=\rg(P_1$) we get $P=P_1P$ and $P_1=PP_1$. Which further yields 
\begin{center}
   $MP=(MP)^*=(MP_1P)^*=(MP_1M^{-1}MP)^*=MPM^{-1}MP_1=MPP_1=MP_1$. 
\end{center}
Pre-multiplying by $M^{-1}$, we obtain $P=P_1$.\\
$(e) \Rightarrow (a)$ Let us assume that there exist unique idempotent matrices $P$, $Q$ such that $MP=(MP)^*$, $\rg(P)=\rg(A)=\rg(Q)$ and $\rg(Q^T)=\rg(A^T$). From the range condition, we can easily obtain 
\begin{equation}\label{eq2.1}
   A=PA=QA=AQ, ~~Q=QA^{(1)}A=AA^{(1)}Q,~~P=AA^{(1)}P. 
\end{equation}
  If we take $X=QA^{(1)}P$, then 
  $AXA=AQA^{(1)}PA=AA^{(1)}A=A$. From  $X=QA^{(1)}P=AA^{(1)}QA^{(1)}P$ and $A=QA=QA^{(1)}A^2=QA^{(1)}PA^2=XA^2$, we obtain $\rg(X)=\rg(A)$. Using equation \eqref{eq2.1}, we obtain 
  \begin{center}
      $X=QA^{(1)}P=QA^{(1)}M^{-1}(MP)^*=QA^{(1)}M^{-1}(MAA^{(1)}P)^*=QA^{(1)}M^{-1}(A^{(1)}P)^*A^*M$
  \end{center}
and 
\begin{center}
    $A^*M=(PA)^*M=A^*(MP)^*=A^*MP=A^*MAA^{(1)}P=A^*MAQA^{(1)}P=A^*MAX$.
\end{center}
Therefore $\rg(X^T)=\rg(A^*M)^T$. The representation of $X=A^{\#}AA^{(1,3^M)}$ can be verified easily and {other} representations follow from the proof of $(e)\Rightarrow (a)$. 
\end{proof}

The relation between $M$-weighted core and $N$-weighted dual core inverse is given in the next result.
\begin{lemma}\label{lem2.8}
Let $A \in \cnn$ and $M \in \cnn$ be an invertible {hermitian} matrix. Then $A^{\core,M}$ exists if and only if $(M^{-1}A^*M)^{M,\core}$ exists. Moreover, $(M^{-1}A^*M)^{M,\core}=M^{-1}[A^{\core,M}]^*M$.
\end{lemma}
\begin{proof}
Let $X=A^{\core,M}$. Then {by Definition \ref{mcore} and Theorem \ref{thm2.7} we have,} $AXA=A$, $XAX=X$, $\rg(X)=\rg(A)$ and $\rg(X^T)=\rg((A^*M)^T)$. Define $Y=M^{-1}X^*M$ and $B=M^{-1}A^*M$. Next we will claim that $Y$ is $M$-weighted dual core inverse of $B$. From 
\begin{center}
    $B^2Y=M^{-1}(A^*)^2X^*M{=M^{-1}\left(XA^2\right)^*M}=M^{-1}A^*M=B$,
\end{center}
\begin{center}
    $Y^2B=M^{-1}(X^*)^2A^*M={M^{-1}\left(AX^2\right)^*M}=M^{-1}X^*M=Y$,
\end{center}
$(MYB)^*=(X^*A^*M)^*=(MAX)=X^*A^*M=MYB$, we have $B^{M,\core}=Y$. Similarly, we can show the converse part. 
\end{proof}
In view of the Theorem \ref{thm2.7} and Lemma \ref{lem2.8}, we have the following equivalent descriptive statement for $N$-weighted dual core {inverse.}
\begin{theorem}\label{thm2.9}
Let $N \in \cnn$ be an {invertible} hermitian matrix and  $A\in \cnn$. Then the following five conditions are equivalent:
\begin{enumerate}
\item [(a)] there {exists} $X\in \cnn$ such that $AXA=A$, $\rg(NX)=\rg(A^*)$, $\rg(X^T)=\rg(A^T)$;
\item [(b)] $AXA=A,$  $XAX=X$,  $(NXA)^*=NXA$, $A^2X=A$, and $X^2A=X$;
\item[(c)] $(NXA)^*=NXA$,$X^2A=X$, and $A^2X=A$;
\item[(d)] $A^{\#}$ exists and $X\in A\{1,4^N\}$;
\item[(e)] there exist unique idempotent matrices $P,~Q \in \cnn$ such that $NQ=(NQ)^*$, $\rg(P)=\rg(A)$ and $\rg(P^T)=\rg(Q^T)=\rg(A^T)$.
\end{enumerate} 
Moreover, If any one of the above $(a)-(e)$ is true, then the $N$-weighted dual core inverse of the matrix $A$ is given by $A^{N,\core} =X=A^{(1,4^N)}A A^{\#}=QA^{(1)}P$. Further $A^{\#}=A\left(A^{N,\core}\right)^2$, $P=AA^{\#}=AA^{N,\core}$, and $QA=A^{N,\core}A$.
\end{theorem}
\begin{corollary}\label{cor2.10}
Let $A\in\cnn$ be a core matrix and $M \in \cnn$ be an invertible hermitian matrices. Further, let $X\in\cnn$ with $(MAX)^*=MAX$. If $X$ satisfies either $XA^2=A$ or $XA=A^{\#}A$, then $A^{\core,M}=A^{\#}AX$.
\end{corollary}
\begin{proof}
Let  $XA^2=A$. Then     $AXA=AXAA^{\#}A=AXA^2A^{\#}=A^2A^{\#}=A$. 
Thus $X\in A\{1,3^M\}$ and hence by Theorem \ref{thm2.7} $(d)$, $A^{\core,M}=A^{\#}AX$. Further, from $XA=A^{\#}A$, we obtain $AXA=A$. Which yields $X\in A\{1,3^M\}$ and hence the proof is complete.
\end{proof}

The following propositions can be used as an equivalent definition for $M$-weighted core inverse.
\begin{proposition}\label{prop11}
Let $M \in \cnn$ be an {invertible} hermitian matrix and  $A\in \cnn$ be a core matrix. If a matrix $X\in\cnn$ satisfies 
\begin{center}
$(3^M)$  $(MAX)^*=MAX$ and $(6)$ $ XA^2= A,$   \end{center}
then $X$ is the $M$-weighted core inverse of $A.$
\end{proposition}

\begin{proof}
Let $A=XA^2$. Then $AA^{\#}= XA^2A^{\#}= XA$. Now ${A}  {X}  {A}= {A}  {A}  {A}^{\#}= {A}.$
 Therefore, $ {X}\in {A}\{1,3^M\}.$
 { So, by Theorem \ref{thm2.7},  we obtain $ {A}^{\core, M}= {X}.$ }
\end{proof}

\begin{proposition}\label{prop12}
Let $M \in \cnn$ be an {invertible} hermitian matrix and  $A\in \cnn$. If a matrix $X$ satisfies
\begin{center}
    $(1)$\  $ {A}  {X}  {A}= {A}$,\ $(3^M)$\  $(M {A}  {X})^*= M{A}  {X}$ and $(7)$\ $ {A}  {X}^2= {X},$ 
\end{center}
then $ {X}$ is the $M$-weighted core inverse of $ {A}.$
\end{proposition}
\begin{proof}
{
From  $ {A}  {X}  {A}= {A}$ and $ {A}  {X}^2= {X}$, we get  $ {A}= {A}^2  {X}^2  {A}.$ Thus $\rg(A)\subseteq \rg(A^2)$ and subsequently $A$ is group invertible.  Hence by Theorem \ref{thm2.7}, ${X}$ is the $M$-weighted core inverse of the matrix $ {A}.$}
\end{proof}

It is well known that the generalized weighted Moore-Penrose inverse does not exist always \cite{SheGuo07}. But {if we replace the invertibility by the positive definiteness for matrices $M$ and $N$,} then $A^\dagger_{M,N}$ always exists and we call it the weighted Moore-Penrose inverse \cite{BenIsrael03} of the matrix $A$. 

\begin{theorem}\label{2.11}
Let $A\in \cnn$ and $M,~N \in \cnn$ be invertible hermitian matrices. Then the following are equivalent:
\begin{enumerate}
\item [(a)] $A^\dagger_{M,N}$ exists;
\item [(b)] there exist unique idempotent matrices $P,~Q \in \cnn$ such that $MP=(MP)^*$, $NQ=(NQ)^*$, $\rg(P)=\rg(A)$ and $\rg(Q^T)=\rg(A^T)$.
\end{enumerate} 
moreover, if any one of the statements $(a)-(b)$ holds, then  $A^\dagger_{M,N}=QA^{(1)}P$. 
\end{theorem}
\begin{proof}
$(a)\Rightarrow (b)$ Let $X=A^\dagger_{M,N}$. If we define $P=AX$ and $Q=XA$, then we can easily show that $\rg(P)=\rg(A)$ and $\rg(Q^T)=\rg(A^T)$. Further $P=AX=AXAX=P^2$, $Q=XA=XAXA=Q^2$, $MP=MAX=(MAX)^*=(MP)^*$, and $NQ=NXA=(NQ)^*$. To show the uniqueness of $P$ and $Q$, assume that {there exist} two idempotent pairs $(P,Q)$ and $(P_1,Q_1)$ which satisfies $(b)$. Now from $\rg(Q^T)=\rg(Q_1^T)$, we have $Q=QQ_1$ and $Q_1=Q_1Q$. Using $Q=QQ_1$ and $Q_1=Q_1Q$, we obtain
\begin{center}
    $NQ=(NQ)^*={(NQN^{-1}NQ_1)^*}=NQ_1N^{-1}NQ=NQ_1Q=NQ_1$.
\end{center}
Pre-multiplying $N^{-1}$, we get $Q=Q_1$. Similarly, we can show the uniqueness of $P$.\\
$(b)\Rightarrow (a)$ Let $P$ and $Q$ be the unique idempotent matrices with $MP=(MP)^*$, $NQ=(NQ)^*$, $\rg(P)=\rg(A)$ and $\rg(Q^T)=\rg(A^T)$. Then $A=PA=AQ$, $P=AU$, and $Q=VA$ for some $U, ~V\in\cnn$. Further, $P=AA^{(1)}AU=AA^{(1)}P$ and $Q=QA^{(1)}A$. Now, consider $X=QA^{(1)}P$. Then $A^\dagger_{M,N}=X=QA^{(1)}P$ is follows from the following verification:
\begin{center}
    $AXA=AQA^{(1)}PA=A,~~XAX=QA^{(1)}PAQA^{(1)}P=X$;
\end{center}
\begin{center}
    $MAX=MAQA^{(1)}P=MP=(MAX)^*,~~NXA=NQA^{(1)}PA=NQ=(NXA)^*$.
\end{center}
\end{proof}

\begin{theorem}\label{thm2.12}
Let $A\in\cnn$ and $M \in \cnn$ be an {hermitian and positive definite matrix}. If $ind(A)=1$ and {$A\{1,3^M\} \neq \phi$,} then the following five conditions are true.
\begin{enumerate}
    \item [(a)] ($A^{\core,M})^{\#}=A^2A^{\core,M}=(A^{\core,M})_{M,M}^{\dagger}=(A^{\core,M})^{\core,M}=(A^{\core,M})^{M,\core}$;
    \item[(b)] ($A^{\#})^{\core,M}=A^2A^{\core,M};$
    \item[(c)] $A^{\#}=(A^{\core,M})^2A$;
    \item[(d)] $(A^n)^{\core,M}=(A^{\core,M})^n$ for any $n\in\mathbb{N}$;
    \item[(e)] $[(A^{\core,M})^{\core,M}]^{\core,M}=A^{\core,M}$.
\end{enumerate}
\end{theorem}
\begin{proof}
$(a)$ From Theorem \ref{thm2.7}, we obtain $A^{\core,M}=A^{\#}AA^{(1,3^M)}$. Now, let $X=A^2A^{\core,M}=A^2A^{\#}AA^{(1,3^M)}=A^2A^{(1,3^M)}$. Then 
\begin{equation}\label{eq21}
    A^{\core,M}XA^{\core,M}=A^{\#}AA^{(1,3^M)}XA^{\#}AA^{(1,3^M)}=A^{\#}AA^{(1,3^M)}=A^{\core,M},
\end{equation}
\begin{equation}\label{eq22}
    XA^{\core,M}X= XA^{\#}AA^{(1,3^M)}X=A^2A^{(1,3^M)}A^{\#}AA^{(1,3^M)}A^2A^{(1,3^M)}=A^2A^{(1,3^M)}=X, 
\end{equation}
 and 
 \begin{equation}\label{eq23}
     A^{\core,M}X=A^{\#}AA^{(1,3^M)}AAA^{(1,3^M)}=A^2A^{\#}A^{(1,3^M)}=AAA^{(1,3^M)}AA^{\#}A^{(1,3^M)}=XA^{\core,M}.
 \end{equation}
From equations \eqref{eq21}-\eqref{eq23}, we obtain $\left(A^{\core,M}\right)^{\#}=X=A^2A^{\core,M}$. Further, we have 
\begin{eqnarray}\label{eq24}
\nonumber
(MA^{\core,M}X)^*&=&(MA^{\#}AA^{(1,3^M)}A^2A^{(1,3^M)})^*=(MAA^{(1,3^M)})^*=MAA^{(1,3^M)}\\
&=&MAA^{\#}AA^{(1,3^M)}=MA^{\#}AA^{(1,3^M)}A^2A^{(1,3^M)}=MA^{\core,M}X,
\end{eqnarray}
 and 
 \begin{eqnarray}\label{eq25}
 \nonumber
 (MXA^{\core,M})^*&=&(MA^2A^{(1,3^M)}A^{\core,M})^*=(MA^2A^{(1,3^M)}A^{\#}AA^{(1,3^M)})^*=(MAA^{(1,3^M)})^*\\
 &=&MAA^{(1,3^M)}=MAA^{\#}AA^{(1,3^M)}=MA^2A^{(1,3^M)}A^{\core,M}=MXA^{\core,M}.
 \end{eqnarray}
 Combining equations \eqref{eq21} and \eqref{eq22} along with \eqref{eq24} and \eqref{eq25}, we get ${\left(A^{\core,M}\right)^\dagger_{M,M}}=X=A^2A^{\core,M}$. Next we will show $(A^{\core,M})^{\core,M}=X$. Using the definition of $M$-weighted core inverse, we have  $A^{\core,M}X^2=(A^{\core,M}A^2)(A^{\core,M}A^2)A^{\core,M}=AAA^{\core,M}=X$, and 
 $X(A^{\core,M})^2=A^2\left(A^{\core,M}\right)^3=A\left(A^{\core,M}\right)^2=A^{\core,M}.$ Hence $(A^{\core,M})^{\core,M}=X$ since $(MA^{\core,M}X)^*=MA^{\core,M}X$. To claim the final part of $(a)$, it is enough to show $\left(A^{\core,M}\right)^2X=A^{\core,M}$ and $X^2A^{\core,M}=X$. Using Proposition \ref{pro2.5} and the definition of $M$-weighted core inverse, we obtain
 \begin{center}
     $\left(A^{\core,M}\right)^2X=\left(A^{\core,M}\right)^2A^2A^{\core,M}=A^{\core,M}AA^{\core,M}=A^{\core,M}$, and
 \end{center}
 \begin{center}
   $X^2A^{\core,M}=A^2A^{\core,M}A^2\left(A^{\core,M}\right)^2=A^2A^{\core,M}AA^{\core,M}=A^2A^{\core,M}=X$.  
 \end{center}
 $(b)$ Let $X=A^2A^{\core,M}$. Then the result is follows from the following expressions
 \begin{center}
     $A^{\#}X^2=A^{\#}A^2A^{\core,M}A^2A^{\core,M}=AAA^{\core,M}=X$,
 \end{center}
 \begin{center}
     $X(A^{\#})^2=A^2A^{\core,M}(A^{\#})^2=A^2A^{\core,M}A^2(A^{\#})^4=A^3(A^{\#})^4=A^{\#}$, and
 \end{center}
 \begin{center}
     $(MA^{\#}X)^*=(MA^{\#}A^2A^{\core,M})^*=(MAA^{\core,M})^*=MAA^{\core,M}=MA^{\#}X$.
 \end{center}
 $(c)$ From Theorem \ref{thm2.7}, the result is trivially holds.\\
 $(d)$ First we will show for $n=2$, i.e.,  $(A^2)^{\core,M}=(A^{\core,M})^2$. Since ${A^2(A^{\core,M})^4}=A(A^{\core,M})^3=(A^{\core,M})^2$, $(A^{\core,M})^2A^4=A^{\core,M}A^3=A^2$, and $(MA^2(A^{\core,M})^2)^*=(MAA^{\core, M})^*=MAA^{\core,M}=MA^2(A^{\core,M})^2$, so it follows that $(A^2)^{\core,M}=(A^{\core,M})^2$. Now assume that the result is true for $n=k$, i.e., $(A^k)^{\core,M}=(A^{\core,M})^k$. Next we will claim for $n=k+1$. From
\begin{center}
    $A^{k+1}(A^{\core,M})^{2k+2}=A^k(A^{\core,M})^{2k+1}=(A^{\core,M})^kA^{\core,M}=(A^{\core,M})^{k+1}$,
\end{center}
\begin{center}
    $(A^{\core,M})^{k+1}A^{2k+2}=(A^{\core,M})^kA^{2k+1}=A^kA=A^{k+1}$, and 
\end{center}
\begin{center}
    $\left(MA^{k+1}(A^{\core,M})^{k+1}\right)^*=\left(MA^{k}(A^{\core,M})^{k}\right)^*=MA^{k}(A^{\core,M})^{k}=MA^{k+1}(A^{\core,M})^{k+1}$, 
\end{center}
we obtain  $(A^{k+1})^{\core,M}=(A^{\core,M})^{k+1}$. Hence {it is true} for any $n\in \mathbb{N}$. \\
$(e)$ From part $(a)$, we have $(A^{\core,M})^{\core,M}=A^2A^{\core,M}$. So it is {sufficient} to show $(A^2A^{\core,M})^{\core,M}=A^{\core,M}$. Now 
$A^2A^{\core,M}(A^{\core,M})^2=A(A^{\core,M})^2=A^{\core,M}$,
 $A^{\core,M}A^2A^{\core,M}A^2A^{\core,M}=AAA^{\core,M}=A^2A^{\core,M}$, and $(MA^2A^{\core,M}A^{\core,M})^*=(MAA(A^{\core,M})^2)^*=(MAA^{\core,M})^*=MAA^{\core,M}=MA^2A^{\core,M}A^{\core,M}$. Thus completes the proof. 
 \end{proof}
Using the method used in proving the above Theorem \ref{thm2.12}, we can establish the following result for $N$-weighted dual core inverse. 
\begin{theorem}\label{thm2.13}
Let $A\in\cnn$ and $N \in \cnn$ be a
{hermitian and positive definite matrix}. If $ind(A)=1$ and {$A\{1,4^N\}\neq \phi$,} then the following holds
\begin{enumerate}
    \item [(a)] $(A^{N,\core})^{\#}=A^{N,\core}A^2=(A^{N,\core})_{N,N}^{\dagger}=(A^{N,\core})^{N,\core}=(A^{N,\core})^{\core,N}$;
    \item[(b)]  $(A^{\#})^{N,\core}=A^{N,\core}A^2$;
    \item[(c)] $A^{\#}=A(A^{N,\core})^2$;
    \item[(d)] $(A^n)^{N,\core}=(A^{N,\core})^n$ for any $n\in\mathbb{N}$;
    \item[(e)] $[(A^{N,\core})^{N,\core}]^{N,\core}=A^{N,\core}$.
\end{enumerate}
\end{theorem}
\begin{corollary}
Let $A\in\cnn$ be a core matrix. If $M, ~N \in \cnn$ are {hermitian and positive definite matrices} and {$A\{1,3^M,4^N\} \neq \phi$,} then the following holds: 
\begin{enumerate}
     \item [(a)] $(A^{\#})_{M,N}^{\dagger}=A^{N,\core }A^3A^{\core,M}$;
    \item[(b)] $A^{\#}=A^{\core,M}AA^{N,\core}$. 
\end{enumerate}
\end{corollary}
\begin{proof}
$(a)$ {Let $X=A^2A^{\core,M}$ and $Y=A^{N,\core}A^2$. 
Then by Theorem \ref{thm2.7}, \ref{thm2.12} $(b)$, and  \ref{thm2.13} we have
\begin{eqnarray*}
&&A^\# X A^\# = A^\# = A^\#YA^\#,\\ 
&&X A^\# X =X,~~Y A^\#Y=Y,\\
&&    (MA^\#X)^*=MAX,~~(NYA^\#)^*=NYA^\#,
\end{eqnarray*}
Now consider $Z=A^{N,\core} A^3A^{\core,M} = YA^{\core}X$. Then $(A^{\#})_{M,N}^{\dagger}=Z$ follows from the following verification:
\begin{eqnarray*}
&&A^{\#}ZA^{\#}=A^{\#}YA^{\#}XA^{\#}=A^{\#}XA^{\#}=A^{\#},\\
&&ZA^{\#}Z=YA^{\#}XA^{\#}YA^{\#}X=YA^{\#}YA^{\#}X=YA^{\#}X=Z,\\
&&(MA^{\#}Z)^*=(MA^{\#}YA^{\#}X)^*=(MA^{\#}X)^*=MA^{\#}X=MA^{\#}Z,\\
&&(NZA^{\#})^*=(NYA^{\#}XA^{\#})^*=(NYA^{\#})^*=NYA^{\#}=NZA^{\#}.
\end{eqnarray*}
}
$(b)$ Using Theorem \ref{thm2.12} $(c)$ and Theorem \ref{thm2.13} $(b)$, we have 
\begin{center}
    $A^{\#}=A^{\#}AA^{\#}=(A^{\core, M})^2A^3(A^{N,\core})^2=A^{\core, M}A^2A^{N,\core}A^{N,\core}=A^{\core, M}AA^{N,\core}$.
\end{center}
\end{proof}
The following example illustrates that generalized Moore-Penrose inverse, group inverse, 
$M$-weighted core inverse and $N$-weighted dual-core inverses are different. 
\begin{example}
Let $A=\begin{bmatrix}
1 & 0 & 1\\
0 & 1 & 0\\
0 & 0 & 0
\end{bmatrix}$ and $M=N=\begin{bmatrix}
1 & 0 & 0\\
0 & 2 & 0\\
0 & 0 & 1
\end{bmatrix}$. Then we can easily verify that $A^{\#}=A$,~{$A^{\dagger}_{M,N}=\begin{bmatrix}
0.5 & 0 & 0\\
0 & 1 & 0\\
0.5 & 0 & 0
\end{bmatrix}$},~ $A^{\core,M}=\begin{bmatrix}
1 & 0 & 0\\
0 & 1 & 0\\
0 & 0 & 0
\end{bmatrix}$, and $A^{N,\core}=\begin{bmatrix}
0.5 & 0 & 0.5\\
0 & 1 & 0\\
0.5 & 0 & 0.5
\end{bmatrix}$.
\end{example}

At this point, one can ask a natural question: under which assumption, all these inverses are the same $?$ The following theorem will give an appropriate answer to this, along with a few representations of weighted-EP matrices.

\begin{theorem}\label{thm2.16}
Let $A\in\cnn$ and $ind(A)=1$. If $M$ is an {hermitian and positive definite matrix}, then the following eight statements are equivalent: 
\begin{enumerate}
\item[(a)] The matrix $A$ is weighted-EP {with respect to $(M,M)$};
\item[(b)] there exist a matrix $X\in \cnn$ such that $(MAX)^*=MAX$, $A^2X=A$, $X^2A=X$;
\item[(c)] there exist a matrix $X\in \cnn$ such that $(MXA)^*=MXA$, $XA^2=A$, $AX^2=X$;
\item[(d)]  $A^{\core,M}=A^{\#}=A^{M,\core}$;
\item[(e)] $A^{\core,M}=A^{\dagger}_{M,M}=A^{M,\core}$;
\item[(f)] $\rg(A^T)\subseteq\rg\left((A^*M)^T\right)$;
\item[(g)] $\rg(A)\subseteq\rg\left(M^{-1}A^*\right)$;
\item[(h)] $\rg(A^*)\subseteq\rg(MA)$.
\end{enumerate}
\end{theorem}
\begin{proof}
$(a)\Leftrightarrow(b)$ Let $A$ be weighted-EP with respect to $(M,M)$. Then $(b)$ is trivially hold if we consider $X=A^{\#}=A^{\dagger}_{M,M}$. Conversely, assume that $(MAX)^*=MAX$, $A^2X=A$, $X^2A=X$ for some $X\in\cnn$. Now  $A=A^2X=A^2X^2A=A^2X^3A^2=YA^2$ where $Y=A^2X^3\in\cnn$. By Lemma \ref{lem1.2}, {we obtain $A^{\#}=YAX=Y^2A=AX^2$. Since $MA^{\#}A=MAA^{\#}=MAAX^2=M(A^2X)X=MAX$ and $MAX$ is hermitian, it follows that $A^{\dagger}_{M,M}=A^{\#}$}.\\
$(a)\Leftrightarrow(c)$ in the similar process as $(a)\Leftrightarrow(b)$\\
$(a)\Leftrightarrow (d)$ From the equivalence of $(a)\Leftrightarrow (b)\Leftrightarrow (c)$, it is trivial that $(a)\Rightarrow (d)$. The converse is follows from the fact that $(MA^{\#}A)^*=(MAA^{\#})^*=(MAA^{\core,M})^*=MAA^{\core,M}=MAA^{\#}=MA^{\#}A$.\\
{$(a)\Leftrightarrow (e)$ Let $A^{\core,M}=A^\dagger_{M,M}$. To claim $A^\dagger_{M,M}=A^{\#}$, it is enough to show $MA^{\#}A$ and $MAA^{\#}$ are hermitian. The hermitian property follows from $(MA^{\#}A)^*=(MAA^{\#})^*=(MA^{\core,M}A^2A^{\#})^*=(MA^{\core,M}A)^*=(MA^\dagger_{M,M}A)^*=MA^\dagger_{M,M}A=MA^{\#}A$. Thus $A^\dagger_{M,M}=A^{\#}$}. Hence $A$ is weighted-EP with respect to $(M,M)$. The converse part follows from the equivalence of $(a)\Leftrightarrow (b)\Leftrightarrow (c)$.\\
$(a)\Leftrightarrow (f)$ Let $A^{\#}=A^\dagger_{M,M}$. Then $A=A^2A^{\#}=A^2A^\dagger_{M,M}=AM^{-1}(MAA^\dagger_{M,M})^*=AM^{-1}(A^\dagger_{M,M})^*A^*M$. Thus $\rg(A^T)\subseteq\rg\left((A^*M)^T\right)$. Conversely, let $\rg(A^T)\subseteq\rg\left((A^*M)^T\right)$. Then $A=ZA^*M$ for some $Z\in\cnn$. This yields $A=ZA^*M=Z(AA^{\#}A)^*M=ZA^*MM^{-1}(AA^{\#})^*M=AM^{-1}(AA^{\#})^*M$. Further, $MAA^{\#}=MA^{\#}A=MA^{\#}AM^{-1}(AA^{\#})^*M$ is hermitian. Hence $A^{\#}=A^\dagger_{M,M}$.\\
$(a)\Leftrightarrow (g)$ Let $A^{\#}=A^\dagger_{M,M}$. Then $A=A^{\#}A^2=M^{-1}(MA^{\#}A)A=M^{-1}(MA^\dagger_{M,M}A)^*A=M^{-1}A^*(MA^\dagger_{M,M})^*A$. Thus $\rg(A)\subseteq\rg(M^{-1}A^*)$. Conversely, let $\rg(A)\subseteq\rg(M^{-1}A^*)$. Then $A=M^{-1}A^*Z$ for some $Z\in\cnn$. Which yields $A=M^{-1}A^*Z=M^{-1}(AA^{\#}A)^*Z=M^{-1}(AA^{\#})^*MM^{-1}A^*Z=M^{-1}(AA^{\#})^*MA$. Further, we obtain $MA^{\#}A=MAA^{\#}=MM^{-1}(AA^{\#})^*MAA^{\#}=(AA^{\#})^*MAA^{\#}$ is hermitian. Therefore, $A^{\#}=A^\dagger_{M,M}$. The equivalence between $(a)\Leftrightarrow (h)$ can be proved in the similar way. 
\end{proof}
{In a similar} manner, we can also show the following result. 
\begin{theorem}\label{thm2.17}
Let $A\in\cnn$ and $ind(A)=1$. If $M$ and $N$ are invertible hermitian matrices, then the following are equivalent:
\begin{enumerate}
\item[(a)] The matrix $A$ is weighted-EP {with respect to $(M,N)$};
\item[(b)] there exist $X\in \cnn$ such that $(MAX)^*=MAX$, $(NAX)^*=NAX$, $A^2X=A$, $X^2A=X$;
\item[(c)] there exist $X\in \cnn$ such that $(MXA)^*=MXA$, $(NXA)^*=NXA$, $XA^2=A$, $AX^2=X$;
\item[(d)]  $\rg(A)\subseteq\rg\left(N^{-1}A^*\right)$ and $\rg(A^T)\subseteq\rg\left((A^*M)^T\right)$;
\item[(e)] $\rg\left((AN^{-1})^T\right)\subseteq\rg((A^*)^T)$ and $\rg(A)\subseteq\rg\left(M^{-1}A^*\right)$;
\item[(f)] $\rg\left(N^{-1}A^*\right)\subseteq\rg(A)$ and $\rg((A^*)^T)\subseteq\rg\left((AM^{-1})^T\right)$;
\item[(g)] $\rg((A^*)^T)\subseteq\rg\left((AN^{-1})^T\right)$ and $\rg(A)\subseteq\rg(AM)$.
\end{enumerate}
\end{theorem}

\begin{theorem}\label{thm2.18}
Let $M$, $N$ be hermitian and positive definite matrices. 
If $A\in\cnn$ be a core matrix, then the following five conditions are equivalent:
\begin{enumerate}
\item[(a)] $A$ is weighted-EP with respect to $(M,N)$;
\item[(b)] $A^nA^{\core,M}=A^{N,\core}A^n$ for $n\in\mathbb{N}$;
\item[(c)] $(A^{\#})^nA^{\core,M}=A^{N,\core}(A^{\#})^n$ for $n\in\mathbb{N}$;
\item[(d)] $\left(A^{\core,M}\right)^{\core,M}=\left(A^{N,\core}\right)^{N,\core}$;
\item[(e)] $A^{\core,M}=A^{N,\core}$.
\end{enumerate}
\end{theorem}
{
\begin{proof}
(a)$\Rightarrow$ (b)--(e)  It  follows from $A^\dagger_{M,N}=A^{\#}=A^{\core,M}=A^{N,\core}$, Theorem \ref{thm2.12} (a), and \ref{thm2.13} (a).\\
(b)$\Rightarrow$ (a) It is enough to show that $MAA^{\#}$ and $NA^{\#}A$ are hermitian. Let $A^nA^{\core,M}=A^{N,\core}A^n$. Now  
$AA^{\core,M}=(A^{\#})^nA^{n+1}A^{\core,M}=(A^{\#})^nAA^{N,\core}A^n=(A^{\#})^nA^n=A^{\#}A=AA^{\#}$. Similarly, we can show that $A^{N,\core}A=A^{\#}A$. The hermitian property of $MAA^{\#}$ and $NA^{\#}A$ holds due to the fact that both $MAA^{\core,M}$ and $NA^{N,\core}A$ are hermitian.\\
(c)$\Rightarrow$ (a)  In similar way (b)$\Rightarrow$ (a).\\
(d)$\Rightarrow$ (a) Let $\left(A^{\core,M}\right)^{\core,M}=\left(A^{N,\core}\right)^{N,\core}$. Then by Theorem \ref{thm2.12} (a), and \ref{thm2.13} (a), we obtain $A^2A^{\core,M}=A^{N,\core}A^2$. Thus   $A^nA^{\core,M}=A^{N,\core}A^n$. Hence the result follows from (b)$\Rightarrow$ (a).\\
(e)$\Rightarrow$ (a) Let $A^{\core,M}=A^{N,\core}$. Then 
\begin{center}
 $AA^{\core,M}=AA^{N,\core}=A\left(A^{N,\core}\right)^2A=A\left(A^{N,\core}\right)^2A^2A^{\#}=AA^{N,\core}AA^{\#}=AA^{\#}$.    
\end{center}
Similarly we can show $A^{N,\core}A=A^{\#}A$. This completes the proof.
\end{proof}

}

\section{Reverse order law for weighted core inverse}
It is well-known that the reverse-order law for matrices are a topic of considerable  research in the theory of generalized inverses. In this section we establish a few necessary and sufficient conditions for the reverse order law to hold in the framework of weighted core inverse. In this context, a sufficient condition of reverse-order law of M-weighted core inverse is discussed in the following theorem.  
\begin{theorem}\label{thm4.1}
Let {$A$ and} $B \in\cnn$. Consider $M\in\cnn$ to be an invertible hermitian {matrix}. If
\begin{equation*}
    A^{\core,M}B=B^{\core,M}A~~~~\textnormal{and}~~~ AA^{\core,M}=BA^{\core,M},
\end{equation*}
 then
\begin{equation*}
\left(AB\right)^{\core,M}=B^{\core,M}A^{\core,M}=\left(A^{\core,M}\right)^2=(A^2)^{\core,M}.
\end{equation*}
\end{theorem}
\begin{proof}
Let $A^{\core,M}B=B^{\core,M}A$. Then
\begin{center}
 $B^{\core,M}A^{\core,M}=(B^{\core,M}A)(A^{\core,M})^2
= A^{\core,M}BA^{\core,M}A^{\core,M}
= A^{\core,M}AA^{\core,M}A^{\core,M}
= (A^{\core,M})^2$.
\end{center}
From Theorem \ref{thm2.12} $(d)$, we get $B^{\core,M}A^{\core,M}= (A^{\core,M})^2
= (A^2)^{\core,M}$. Next we claim that, $B^{\core,M}A^{\core,M}$ is the $M$-weighted core inverse of $AB$. To show this, let $X=B^{\core,M}A^{\core,M}$. Then
\begin{eqnarray*}
  (MABX)^*&=&(MABB^{\core,M}A^{\core,M})^*=\left(MA(BA^{\core,M})A^{\core,M}\right)^*=(MAAA^{\core,M}A^{\core,M})^*\\
  &=&(MAA^{\core,M})^*=MAA^{\core,M}=MABX,  
\end{eqnarray*}
\begin{eqnarray*}
X(AB)^2&=&B^{\core,M}A^{\core,M}(AB)^2=(A^{\core,M})^2A(BA)B={(A^{\core,M})^2ABA^{\core,M}A^2B}\\
&=&(A^{\core,M})^2A^2A^{\core,M}A^2B={(A^{\core,M})^2A^3B}=AB,\mbox{ and }
\end{eqnarray*}
\begin{eqnarray*}
ABX^2&=&AB(B^{\core,M}A^{\core,M})^2=AB(A^{\core,M})^4=A^2(A^{\core,M})^4=(A^{\core,M})^2=B^{\core,M}A^{\core,M}.
\end{eqnarray*}
Therefore, $(AB)^{\core,M}=B^{\core,M}A^{\core,M}$.
\end{proof}

Similarly, we can show the following result for $N$-weighted dual core inverse. 
\begin{theorem}\label{thm4.2}
Let $N\in\cnn$ be an invertible hermitian matrix and $ A,B \in\cnn$. If
\begin{equation*}
AB^{N,\core}=BA^{N,\core} ~~~~\textnormal{and}~~~~ B^{N,\core}B=B^{N,\core}A, 
\end{equation*}
then
\begin{equation*}
(AB)^{N,\core}=B^{N,\core}A^{N,\core}=(B^{N,\core})^2=(B^2)^{N,\core}.
\end{equation*}
\end{theorem}

\begin{theorem}\label{unit} 
Let $M\in\cnn$ be an invertible hermitan matrix and $ A,B \in\cnn$ with $ind(A)=1=ind(B)=ind( {A}  {B})$.
\begin{enumerate}
    \item[(a)] If the matrix $ {B}$ is unitary and $\rg( {B}^*  {A}^{\core, M})\subseteq\rg( {A}^{\core, M})$, then 
    \begin{equation*}
        ( {A}  {B})^{\core, M}= {B}^*  {A}^{\core, M}.
    \end{equation*}
    \item[(b)] If the matrix $ {A}$ is unitary and $\rg( {A})\subseteq\rg(B)$, then 
    \begin{equation*}
    ( {A}  {B})^{\core, M}= {B}^{\core, M}  {A}^{*}.    
    \end{equation*}
    
\end{enumerate}
\end{theorem}

\begin{proof}
$(a)$ Let $\rg( {B}^*  {A}^{\core, M})\subseteq\rg( {A}^{\core, M})$. This implies $ {B}^*  {A}^{\core, M}= {A}^{\core, M}  {U}$ for some $ {U}\in\cnn$. Now
\begin{center}
   $ {A}  {B}  {B}^*  {A}^{\core, M}  {A}  {B}= {A}  {A}^{\core, M}  {A}  {B}= {A}  {B}$,  \\
   $( M{A}  {B}  {B}^*  {A}^{\core, M})^*=( M{A}  {A}^{\core, M})^*= M{A}  {A}^{\core, M}= M{A}  {B}  {B}^*  {A}$, and 
\end{center}
\begin{eqnarray*}
 {A}  {B} ( {B}^*  {A}^{\core, M})^2= {A}  {A}^{\core, M}  {B}^*  {A}^{\core, M}= {A}  {A}^{\core, M}  {A}^{\core, M}  {U}= {A}^{\core, M}  {U}= {B}^*  {A}^{\core, M}.
\end{eqnarray*}
Thus by Proposition \ref{prop12}, $( {A}  {B})^{\core, M}= {B}^*  {A}^{\core, M}$.\\
$(b)$ Let $\rg( {A})\subseteq\rg( {B})$. Then $ {A}= {B}  {U}$ for some $ {U}\in\cnn$, 
\begin{center}
    $ {B}^{\core, M}  {A}^{*} (  {A}  {B})^2= {B}^{\core, M}  {B}  {A}  {B}= {B}  {U}  {B}= {A}  {B}$, and {$A=BU=B{B}^{\core, M}  {B}U=B{B}^{\core, M}A$.}
\end{center} 
Further, 

{
\begin{eqnarray*}
(MAB{B}^{\core, M}A^*)^*&=& (MAM^{-1}MBB^{\core, M}A^{*})^*
=AMBB^{\core, M}M^{-1}A^*M\\
&=&AMBB^{\core, M}AA^*M^{-1}A^*M
=AMBB^{\core, M}BUA^*M^{-1}A^*M\\
&=&AMBUA^*M^{-1}A^*M
=AMAA^*M^{-1}A^*M.\\
&=&AMM^{-1}A^*M
=M\\
&=&MA(A)A^*A^*
=MABB^{\core, M}AA^*A^*
=MABB^{\core, M}A^*
\end{eqnarray*}
}

Hence by Proposition \ref{prop11}, $(AB)^{\core, M}= B^{\core}A^{*}$.
\end{proof}

\section{Conclusion}
In the literature, the weighted core inverses for a matrix are defined in a particular manner (see  \cite{fer18,gao2018,MaH19}). Our research introduces new representations for the $M$-weighted core and $N$-weighted dual-core inverses for matrices. Subsequently, we investigate the properties of weighted core inverse of matrices and discuss the existence of generalized weighted Moore-Penrose inverse along with the relationships with other generalized inverses. The beauty of the new representation of weighted core inverses is that the algorithm developed for one particular choice of generalized inverse can be easily extended to the other types of generalized inverses. This gives us the flexibility to choose generalized inverses depending on applications. In addition to that, we have also discussed the conditions under which the reverse order laws hold for these inverses.

\medskip

\noindent{\bf{Acknowledgments.}}\\
The first and the third authors are grateful to the Mohapatra Family Foundation and the College of Graduate Studies, University of Central Florida, Orlando, for their financial support for this research.

\bibliographystyle{abbrv}
\bibliography{reference}

\begin{thebibliography}{10}

\bibitem{BakTr10}
O.~M. Baksalary and G.~Trenkler.
\newblock Core inverse of matrices.
\newblock {\em Linear Multilinear Algebra}, 58(5-6):681--697, 2010.

\bibitem{BakTr14}
O.~M. Baksalary and G.~Trenkler.
\newblock On a generalized core inverse.
\newblock {\em Appl. Math. Comput.}, 236:450--457, 2014.

\bibitem{BapatRao90}
R.~B. Bapat, K.~P. S.~B. Rao, and K.~M. Prasad.
\newblock Generalized inverses over integral domains.
\newblock {\em Linear Algebra Appl.}, 140:181--196, 1990.

\bibitem{baskett1969}
T.~S. Baskett and I.~J. Katz.
\newblock Theorems on products of {$EP_{r}$} matrices.
\newblock {\em Linear Algebra and Appl.}, 2:87--103, 1969.

\bibitem{BenIsrael03}
A.~Ben-Israel and T.~N.~E. Greville.
\newblock {\em Generalized inverses: theory and applications}.
\newblock Wiley-Interscience [John Wiley \& Sons], New York-London-Sydney, 1974.
\newblock Pure and Applied Mathematics.

\bibitem{Cline68}
R.~E. Cline.
\newblock Inverses of rank invariant powers of a matrix.
\newblock {\em SIAM J. Numer. Anal.}, 5:182--197, 1968.

\bibitem{deng2011}
C.~Y. Deng.
\newblock Reverse order law for the group inverses.
\newblock {\em J. Math. Anal. Appl.}, 382(2):663--671, 2011.

\bibitem{Drazin58}
M.~P. Drazin.
\newblock Pseudo-inverses in associative rings and semigroups.
\newblock {\em Amer. Math. Monthly}, 65:506--514, 1958.

\bibitem{fer18}
D.~E. Ferreyra, F.~E. Levis, and N.~Thome.
\newblock Revisiting the core {EP} inverse and its extension to rectangular matrices.
\newblock {\em Quaest. Math.}, 41(2):265--281, 2018.

\bibitem{GaoChen18}
Y.~Gao and J.~Chen.
\newblock Pseudo core inverses in rings with involution.
\newblock {\em Comm. Algebra}, 46(1):38--50, 2018.

\bibitem{gao2018}
Y.~Gao, J.~Chen, and P.~Patr{\'\i}cio.
\newblock Representations and properties of the {W}-weighted core-{EP} inverse.
\newblock {\em Linear and Multilinear Algebra}, pages 1--15, 2018.

\bibitem{Grev1966}
T.~N.~E. Greville.
\newblock Note on the generalized inverse of a matrix product.
\newblock {\em SIAM Rev.}, 8:518--521; erratum, ibid. 9 (1966), 249, 1966.

\bibitem{khatri1968generalized}
C.~Khatri and C.~Rao.
\newblock Generalized inverse of matrices and its applications, 1968.

\bibitem{Katri68}
C.~G. Khatri.
\newblock Some results for the singular normal multivariate regression models.
\newblock {\em Sankhy\={a} Ser. A}, 30:267--280, 1968.

\bibitem{kurata}
H.~Kurata.
\newblock Some theorems on the core inverse of matrices and the core partial ordering.
\newblock {\em Appl. Math. Comput.}, 316:43--51, 2018.

\bibitem{MaH19}
H.~Ma.
\newblock A characterization and perturbation bounds for the weighted core-{EP} inverse.
\newblock {\em Quaestiones Mathematicae}, pages 1--11, 2019.

\bibitem{HaiTing19}
H.~Ma and T.~Li.
\newblock Characterizations and representations of the core inverse and its applications.
\newblock {\em Linear and Multilinear Algebra}, pages 1--11, 2019.

\bibitem{ManMo14}
K.~Manjunatha~Prasad and K.~S. Mohana.
\newblock Core-{EP} inverse.
\newblock {\em Linear Multilinear Algebra}, 62(6):792--802, 2014.

\bibitem{mat}
G.~Matsaglia and G.~P.~H. Styan.
\newblock Equalities and inequalities for ranks of matrice.
\newblock {\em Linear Multilinear Algebra}, 2:3:269--292, 1974.

\bibitem{mosic2019}
D.~Mosi{\'c}.
\newblock Weighted core--{EP} inverse of an operator between {H}ilbert spaces.
\newblock {\em Linear and Multilinear Algebra}, 67(2):278--298, 2019.

\bibitem{DijanaChu18}
D.~Mosi\'{c}, C.~Deng, and H.~Ma.
\newblock On a weighted core inverse in a ring with involution.
\newblock {\em Comm. Algebra}, 46(6):2332--2345, 2018.

\bibitem{MosiDj11}
D.~Mosi\'{c} and D.~S. Djordjevi\'{c}.
\newblock Reverse order law for the {M}oore-{P}enrose inverse in {$C^*$}-algebras.
\newblock {\em Electron. J. Linear Algebra}, 22:92--111, 2011.

\bibitem{MosiDij12}
D.~Mosi\'{c} and D.~S. Djordjevi\'{c}.
\newblock Reverse order law for the group inverse in rings.
\newblock {\em Appl. Math. Comput.}, 219(5):2526--2534, 2012.

\bibitem{PaniBeheMi20}
K.~Panigrahy, R.~Behera, and D.~Mishra.
\newblock Reverse-order law for the {M}oore--{P}enrose inverses of tensors.
\newblock {\em Linear Multilinear Algebra}, 68(2):246--264, 2020.

\bibitem{PrasadBapat92}
K.~M. Prasad and R.~B. Bapat.
\newblock The generalized {M}oore-{P}enrose inverse.
\newblock {\em Linear Algebra Appl.}, 165:59--69, 1992.

\bibitem{PrasadBapa91}
K.~M. Prasad, K.~P. S.~B. Rao, and R.~B. Bapat.
\newblock Generalized inverses over integral domains. {II}. {G}roup inverses and {D}razin inverses.
\newblock {\em Linear Algebra Appl.}, 146:31--47, 1991.

\bibitem{rakic}
D.~S. Raki\'{c}, N.~v. Din\v{c}i\'{c}, and D.~S. Djordjevi\'{c}.
\newblock Group, {M}oore-{P}enrose, core and dual core inverse in rings with involution.
\newblock {\em Linear Algebra Appl.}, 463:115--133, 2014.

\bibitem{rakic2015}
D.~S. Raki{\'c} and D.~S. Djordjevi{\'c}.
\newblock Star, sharp, core and dual core partial order in rings with involution.
\newblock {\em Applied Mathematics and Computation}, 259:800--818, 2015.

\bibitem{raob}
A.~R. Rao and P.~Bhimasankaram.
\newblock {\em Linear Algebra}.
\newblock Hindustan Book Agency, New Delhi, 2000.

\bibitem{RaoMitra71}
C.~R. Rao and S.~K. Mitra.
\newblock {\em Generalized inverse of matrices and its applications}.
\newblock John Wiley \& Sons, Inc., New York-London-Sydney, 1971.

\bibitem{RaoRao98}
C.~R. Rao and M.~B. Rao.
\newblock {\em Matrix algebra and its applications to statistics and econometrics}.
\newblock World Scientific Publishing Co., Inc., River Edge, NJ, 1998.

\bibitem{JR_rev}
J.~K. Sahoo and R.~Behera.
\newblock Reverse-order law for core inverse of tensors.
\newblock {\em Comput. Appl. Math.}, 97(37):Paper No. 9, 2020.

\bibitem{SahBe20}
J.~K. Sahoo, R.~Behera, P.~S. Stanimirovi\'{c}, V.~N. Katsikis, and H.~Ma.
\newblock Core and core-{EP} inverses of tensors.
\newblock {\em Comput. Appl. Math.}, 39(1):Paper No. 9, 2020.

\bibitem{SheGuo07}
X.~Sheng and G.~Chen.
\newblock The generalized weighted {M}oore-{P}enrose inverse.
\newblock {\em J. Appl. Math. Comput.}, 25(1-2):407--413, 2007.

\bibitem{PredragKM17}
P.~S. Stanimirovi\'{c}, V.~N. Katsikis, and H.~Ma.
\newblock Representations and properties of the {$W$}-weighted {D}razin inverse.
\newblock {\em Linear Multilinear Algebra}, 65(6):1080--1096, 2017.

\bibitem{PreMo20}
P.~S. Stanimirovi{\'c}, D.~Mosi{\'c}, and H.~Ma.
\newblock New classes of more general weighted outer inverses.
\newblock {\em Linear and Multilinear Algebra}, pages 1--26, 2020.

\bibitem{YiminWei98}
W.~Sun and Y.~Wei.
\newblock Inverse order rule for weighted generalized inverse.
\newblock {\em SIAM J. Matrix Anal. Appl.}, 19(3):772--775, 1998.

\bibitem{SunWei02}
W.~Sun and Y.~Wei.
\newblock Triple reverse-order law for weighted generalized inverses.
\newblock {\em Appl. Math. Comput.}, 125(2-3):221--229, 2002.

\bibitem{TianWang11}
Y.~Tian and H.~Wang.
\newblock Characterizations of {EP} matrices and weighted-{EP} matrices.
\newblock {\em Linear Algebra Appl.}, 434(5):1295--1318, 2011.

\bibitem{WangYimin18}
G.~Wang, Y.~Wei, and S.~Qiao.
\newblock {\em Generalized inverses: theory and computations}, volume~53 of {\em Developments in Mathematics}.
\newblock Springer, Singapore; Science Press Beijing, Beijing, second edition, 2018.

\bibitem{WangLi19}
H.~Wang, N.~Li, and X.~Liu.
\newblock The $m$-core inverse and its applications.
\newblock {\em Linear and Multilinear Algebra}, 0(0):1--19, 2019.

\bibitem{wang2015}
H.~Wang and X.~Liu.
\newblock Characterizations of the core inverse and the core partial ordering.
\newblock {\em Linear Multilinear Algebra}, 63(9):1829--1836, 2015.

\bibitem{ZhangWei16}
X.~Zhang, Q.~Xu, and Y.~Wei.
\newblock Norm estimations for perturbations of the weighted {M}oore-{P}enrose inverse.
\newblock {\em J. Appl. Anal. Comput.}, 6(1):216--226, 2016.

\bibitem{ZhangXu17}
X.~Zhang, S.~Xu, and J.~Chen.
\newblock Core partial order in rings with involution.
\newblock {\em Filomat}, 31(18):5695--5701, 2017.

\bibitem{zhou2019core}
M.~Zhou, J.~Chen, and D.~Wang.
\newblock The core inverses of linear combinations of two core invertible matrices.
\newblock {\em Linear and Multilinear Algebra}, pages 1--17, 2019.

\end{thebibliography}
\end{document}